\documentclass[11pt,a4paper,twoside]{amsart}

\usepackage[T1]{fontenc}
\usepackage[sc]{mathpazo}
\usepackage{amscd}
\usepackage{amsfonts}
\usepackage{amsmath}
\usepackage{amssymb}
\usepackage{amstext}
\usepackage{amsthm}
\usepackage{enumerate}
\usepackage{fancyhdr}
\usepackage{setspace}
\usepackage[all,cmtip]{xy}
\usepackage{wasysym}

\usepackage{hyperref}

\theoremstyle{definition}
\newtheorem{definition}{Definition}[section]
\newtheorem{notation}[definition]{Notation}

\newtheorem{problem}[definition]{Problem}
\newtheorem{rem}[definition]{Remark}
\newtheorem{example}[definition]{Example}

\theoremstyle{plain}
\newtheorem{thm}[definition]{Theorem}
\newtheorem{lemma}[definition]{Lemma}
\newtheorem{prop}[definition]{Proposition}
\newtheorem{cor}[definition]{Corollary}

\newcommand{\nodetails}[2]{ \ifthenelse{\not \boolean{details}}{#1}{#2} }

\newcommand{\cI}{\mathcal{I}}
\newcommand{\cJ}{\mathcal{J}}

\newcommand{\cO}{\mathcal{O}}
\newcommand{\cP}{\mathcal{P}}

\newcommand{\fp}{\mathfrak{p}}
\newcommand{\fq}{\mathfrak{q}}

\newcommand{\uR}{\underline{R}}

\newcommand{\sphere}{\mathbb{S}}

\DeclareMathOperator*{\hocolim}{hocolim}

\newcommand{\Loc}[2]{#1\lbrack #2^{-1} \rbrack}

\newcommand{\Primes}{P} %family of primes
\newcommand{\Zp}{\Z_{(p)}} %P-local integers
\newcommand{\ZP}{\Z_{(\Primes)}} %P-local integers
\newcommand{\Spec}{\mathrm{Spec}} %Zariski spectrum

\newcommand{\A}[1]{A(#1)}

\newcommand{\AP}[1]{A(#1)_{(\Primes)}}

\newcommand{\Res}{\mathrm{Res}}

\newcommand{\RU}[1]{RU(#1)}
\newcommand{\RUp}[1]{RU(#1)_{(p)}}
\newcommand{\RUP}[1]{RU(#1)_{(P)}}
\newcommand{\RO}[1]{RO(#1)}

\newcommand{\ROP}[1]{RO(#1)_{(P)}}
\newcommand{\RX}[1]{R(#1)}
\newcommand{\RXP}[1]{R(#1)_{(P)}}
\newcommand{\lin}{\mathit{lin}} %linearization map A(G) \to RO(G)

\newcommand{\KUG}{KU_G}
\newcommand{\KUGp}{(KU_G)_{(p)}}
\newcommand{\KUGP}{(KU_G)_{(\Primes)}}
\newcommand{\KOG}{KO_G}

\newcommand{\KOGP}{(KO_G)_{(\Primes)}}
\newcommand{\KXG}{K_G}

\newcommand{\KXGP}{(K_G)_{(\Primes)}}

\newcommand{\CC}[1]{#1 / \!{\sim}} %conjugacy classes of elements of a group

\newcommand{\OP}[1]{O^\Primes(#1)}

\newcommand{\pprime}[1]{#1_{p^\perp}}
\newcommand{\qprime}[1]{#1_{q^\perp}}
\newcommand{\argprime}[2]{#1_{#2^\perp}}
\newcommand{\Ppart}[1]{#1_{\Primes}}
\newcommand{\Pprime}[1]{#1_{\Primes^\perp}}
\newcommand{\sfield}{\mathbb{F}} %minimal splitting field of group G
\newcommand{\sint}{\mathcal{O}_\sfield} %ring of integers in minimal splitting field
\newcommand{\sintp}{\mathcal{O}_{\sfield,(p)}} %ring of p-local integers in minimal splitting field
\newcommand{\sintP}{\mathcal{O}_{\sfield,(\Primes)}} %ring of P-local integers in minimal splitting field
\newcommand{\Qpg}[2]{Q(#1, #2)} %generic prime in P-local rep ring

\newcommand{\IL}{\mathcal{I}_L}
\newcommand{\IC}{\mathcal{I}_C}
\newcommand{\Icyc}{\mathcal{I}_{cyc}}

\newcommand{\OC}{\mathcal{O}_C}
\newcommand{\Ocyc}{\mathcal{O}_{cyc}}

\newcommand{\Q}{\mathbb{Q}}
\newcommand{\Z}{\mathbb{Z}}

\newcommand{\one}{\mathbb{1}}
\newcommand\restr[2]{{% we make the whole thing an ordinary symbol
\left.\kern-\nulldelimiterspace % automatically resize the bar with \right
#1 % the function
\vphantom{\big|} % pretend it's a little taller at normal size
\right|_{#2} % this is the delimiter
}}

\fancypagestyle{headings}{\pagestyle{fancy}}

\fancypagestyle{headings}{
\fancyhead[LE,RO]{\thepage}
\fancyhead[LO,RE]{}
\fancyhead[CE]{Benjamin B\"ohme}
\fancyhead[CO]{Idempotent characters and equivariantly multiplicative splittings of K-theory}
\fancyfoot[C]{} % except the center

}

\usepackage{ifthen}
\newboolean{details}
\setboolean{details}{false}

\usepackage[utf8]{inputenc}
\usepackage{color}
\usepackage{geometry} %decreases the margin size of the ``amsart'' format
\usepackage{setspace}
\author{Benjamin B\"{o}hme}
\address{Max Planck Institute for Mathematics\\
Vivatsgasse 7\\
53111 Bonn\\
Germany}
\email{boehme@mpim-bonn.mpg.de}
\title{Idempotent characters and equivariantly \\ multiplicative splittings of K-theory}

\begin{document}

\begin{abstract}
We classify the primitive idempotents of the $p$-local complex representation ring of a finite group $G$ in terms of the
cyclic subgroups of order prime to $p$ and show that they all come from idempotents of the Burnside ring. Our results hold without
adjoining roots of unity or inverting the order of $G$, thus extending classical structure theorems.
We then derive explicit group-theoretic obstructions for tensor induction to be compatible with the resulting
idempotent splitting of the representation ring Mackey functor.\\
Our main motivation is an application in homotopy theory: we conclude that the idempotent summands of
$G$-equivariant topological $K$-theory and the corresponding summands of the $G$-equivariant sphere spectrum
admit exactly the same flavors of equivariant commutative ring structures, made precise in terms of Hill-Hopkins-Ravenel norm maps.\\
% This paper is a sequel to the author's earlier work on multiplicative induction for the Burnside ring and the sphere spectrum, see
% \href{https://arxiv.org/abs/1802.01938}{arXiv:1802.01938v2}.
\end{abstract}

\pagestyle{headings}
\maketitle
\thispagestyle{empty}
\nodetails{}{\tableofcontents}

%Section 1: Introduction
\section{Introduction}
The purpose of this paper is twofold: We first classify the primitive idempotents in the real and complex
representation rings $\RO{G}$ and $\RU{G}$ of a finite group $G$ and their local variants, as summarized in 
§\ref{subsect:intro idpts}, extending various classical results. We then study the compatibility of tensor
induction with the splittings of $\RO{G}$ and $\RU{G}$ into idempotent summands, and as a consequence
obtain an explicit description of the $G$-equivariant commutative ring spectrum structures occuring as
idempotent summands of real and complex $G$-equivariant topological $K$-theory. See §\ref{subsect:intro mult}
for a summary of these results.

We begin with some motivation. Multiplicative induction is a familiar tool in representation theory and group
cohomology. In the wake of Hill, Hopkins and Ravenel's ground-breaking solution to the Kervaire invariant
one problem \cite{HHR}, it has also received much interest in equivariant homotopy theory.
Starting from the observation that localization can destroy some of the structure of an equivariant
commutative ring spectrum, Hill and Hopkins \cite{HH:EqvarMultClosure} gave a necessary and sufficient
criterion (cf.~Proposition~\ref{prop translation White HH}) for the localization
\[ \Loc{R}{x} := \hocolim \left( R \stackrel{x}{\longrightarrow} S^{-V} \wedge R
\stackrel{x}{\longrightarrow} S^{-(V \oplus V)} \wedge R \stackrel{x}{\longrightarrow} \ldots \right) \]
of a $G$-$E_\infty$ ring spectrum $R$ at an element $x \in \pi_V^G(R)$ to admit a $G$-$E_\infty$ ring
structure. The critical part is that $\Loc{R}{x}$ might not admit Hill-Hopkins-Ravenel \emph{norm maps}
\[ N_K^H \colon G_+ \wedge_H \bigwedge_{H/K} \Res^G_K(R) \to R \]
for all nested subgroups $K \leq H \leq G$. Subsequently, more general notions of equivariant commutative ring
spectra equipped with incomplete collections of norm maps, called $N_\infty$ \emph{ring spectra}, were studied
by Blumberg and Hill in \cite{BH:OperMult}, \cite{BH:ITF} and \cite{BH:modules}.

Interesting examples of equivariant localizations arise from primitive\footnote{An idempotent is
\emph{primitive} if it cannot be written as a sum of non-zero idempotents.} idempotent elements $e \in
\pi_0^G(R)$. These induce a decomposition of the homotopy Mackey functor $\underline{\pi}_*(R)$ into
indecomposable summands (also called \emph{blocks}) of the form
\[ e \cdot \underline{\pi}_*(R) \cong \Loc{\underline{\pi}_*(R)}{e} \]
and hence yield a block decomposition of $R$ as a wedge of $G$-spectra $\Loc{R}{e}$. One can now ask about the
possible $N_\infty$ ring structures on these blocks. Hill and Hopkins' aforementioned criterion involves checking
relations involving multiplicative induction in $\pi_0^G(R)$, which in general are hard to access.

\begin{problem} \label{problem}
Determine the nested subgroups $K \leq H \leq G$ such that
\begin{enumerate}[(1)]
  \item the norm map $N_K^H$ for $R$ descends to a well-defined norm map\footnote{Throughout the paper, we
  write $\tilde{N}$ for the norms of a localization to distinguish them from the norms of the original object.}
  \[ \tilde{N}_K^H \colon G_+ \wedge_H \bigwedge_{H/K} \Res^G_K(\Loc{R}{e}) \to \Loc{R}{e} \]
  on the block of $R$ defined by the primitive idempotent $e \in \pi_0^G(R)$
  \item the induced norm operation on homotopy groups $N_K^H \colon \pi_0^K(R) \to \pi_0^H(R)$ descends to a
  well-defined norm operation
  \[ \tilde{N}_K^H \colon \pi_0^K(\Loc{R}{e}) \to \pi_0^H(\Loc{R}{e}). \]
\end{enumerate}
% there is a norm map\footnote{Throughout the paper, we write $\tilde{N}$ for the norms of a localization to
% distinguish them from the norms of the original object.} $\tilde{N}_K^H$ for the block $\Loc{R}{e}$. Determine those $K \leq H \leq G$ such that a norm
% operations $\tilde{N}_K^H$ can be defined for the block $\Loc{\underline{\pi}_*(R)}{e}$.
\end{problem}

In the prequel \cite{boehme:mult-idempot}, the author gave an explicit group-theoretical
answer in the fundamental example of the $G$-equivariant sphere spectrum $\sphere$. It built on an
analysis of multiplicative induction in the \emph{Burnside ring} $\A{G}$ and Segal's identification
$\pi_0^G(\sphere) \cong \A{G}$ \cite{segal:ESHT}.

In the present paper, we present a complete solution to Problem~\ref{problem} for
%some other $G$-spectra of interest:
$G$-\emph{equivariant complex topological} $K$-theory $\KUG$ and its real analogue $\KOG$. The
homotopy groups
\[ \pi_0^G(\KUG) \cong \RU{G}, \quad \pi_0^G(\KOG) \cong \RO{G} \]
identify with the \emph{complex} and \emph{real representation ring} $\RU{G}$ and $\RO{G}$, respectively, see
e.g.~\cite[§2]{segal:KU_G}.

%+++++++++++++++++++++++++++++++++++++++++++++++++++++++++
\subsection{Primitive idempotents in representation rings} \label{subsect:intro idpts}
Dress' classification of primitive idempotents in the Burnside ring and its local variants
\cite{dress:solvable} was the starting point for the investigation of the idempotent splittings of $\A{G}$
and $\sphere$ in \cite{boehme:mult-idempot}.
Given a collection $\Primes$ of prime numbers, write $\AP{G} := \A{G} \otimes \ZP$ for the $\Primes$-local Burnside ring,
where $\ZP := \Z \left[ p^{-1} \, | \, p \notin \Primes \right]$. Dress showed that the primitive idempotent elements $e_L \in
\AP{G}$ are in canonical bijection with the conjugacy classes of $\Primes$-perfect subgroups $L \leq G$. See §~\ref{subsect:idpts
of A(G)} for further details.

It is known that the complex representation ring $\RU{G}$ has no idempotents other than zero or one, see
\cite[§11.4, Corollary]{serre:linear-reps}. We extend this result to a classification of the primitive
idempotents in the $\Primes$-local representation ring $\RUP{G} := \RU{G} \otimes \ZP$ as follows. Consider
the ``linearization'' map
\[ \lin \colon \AP{G} \to \RUP{G} \]
given by sending a finite $G$-set to its associated permutation representation.

\begin{thm} \label{intro thm idempotents RU}
The assignment $C \mapsto \lin(e_C)$ defines a bijection between the conjugacy classes of cyclic subgroups $C \leq G$ of
order not divisible by any prime in $\Primes$ and the primitive idempotent elements of the ring $\RUP{G}$. Here, $e_C \in \AP{G}$
denotes Dress' idempotent associated to $C$, see Theorem~\ref{Dress idempotents}.
\end{thm}

Theorem~\ref{intro thm idempotents RU} is an instance of the phenomenon that one passes from the Burnside ring
to the representation ring by restricting attention to cyclic subgroups. The proof is given in
§\ref{section:idpt}.

\begin{rem}
Theorem~\ref{intro thm idempotents RU} extends classical work in the following way:
Building on work by Solomon \cite{solomon:burnside}, Gluck \cite{gluck:bound} studies the idempotents $\lin(e_C)$ and their
character values in the rational and the $p$-local case for a single prime $p$, but does not show that they are primitive.
He also observes that Dress' idempotent $e_L \in \AP{G}$ is in the kernel of the linearization map if $L$ is
not a cyclic group; we prove this in the general $\Primes$-local case in Corollary~\ref{cor characters of
burnside images}.
\end{rem}

% \begin{rem}
% Theorem~\ref{intro thm idempotents RU} is of course not true for arbitrary
% algebras over the ring $\AP{G}$. We will see in Remark~\ref{rem idempotents of RU as sums of splitting idempotents} that the
% primitive idempotents of $\RUP{G}$ decompose into sums of idempotents in extensions obtained by adjoining sufficiently many roots
% of unity.
% \end{rem}

We record an immediate consequence of Theorem~\ref{intro thm idempotents RU}. Write $\ROP{G}$ for the
$\Primes$-local real representation ring and $R\Q(G) \otimes \ZP$ for the ring of $\ZP$-linear combinations
of $G$-representations over the rational numbers. It is well-known that these embed into $\RUP{G}$ as subrings.

\begin{cor} \label{intro cor subrings}
The primitive idempotents of $\RUP{G}$ all lie in the subrings $\ROP{G}$ and $R\Q(G) \otimes \ZP$.
Hence, they are precisely the primitive idempotents of these subrings.
% More generally, for any subfield $K$ of $\C$, the primitive idempotents of $\RUP{G}$ lie in the subring $R_K(G) \otimes \ZP$ of
% $\ZP$-linear combinations of $G$-representations over $K$.
\end{cor}

In the special case of $R\Q(G) \otimes \Q$, this result appeared as \cite[Thm.~3]{solomon:burnside}.

% \begin{rem} \label{intro rem splittings of rep rings agree}
% Obviously, the ring localization $\Loc{\RUP{G}}{(\lin(e_C))}$ agrees with the $\AP{G}$-module localization
% %
% \[ \Loc{\RUP{G}}{e_C} \cong \RUP{G} \otimes_{\AP{G}} \Loc{\AP{G}}{e_C}. \]
% %
% Consequently, the blocks of $\Primes$-local $G$-equivariant $K$-theory are given as the $e_C$-localization
% %
% \[ \Loc{\KUGP}{(\lin(e_C))} \simeq \KUGP \wedge \Loc{\sphere_{(\Primes)}}{e_C} \]
% %
% in genuine $G$-spectra. The same is true for $\ROP{G}$ and $\KOGP$.
% \end{rem}

% \begin{cor} \label{intro cor splittings of rep rings agree}
% We may identify $\Loc{\RUP{G}}{(e_C')}$ with the localization
% %
% \[ \Loc{\RUP{G}}{e_C} \cong \RUP{G} \otimes_{\AP{G}} \Loc{\AP{G}}{e_C} \]
% %
% of $\RUP{G}$ viewed as a module over $\AP{G}$. Consequently, the blocks of $G$-equivariant $K$-theory are
% given $\Primes$-locally as the $e_C$-localization
% %
% \[ \Loc{\KUGP}{e_C'} \simeq \KUGP \wedge \Loc{\sphere_{(\Primes)}}{e_C} \]
% %
% in genuine $G$-spectra. The same is true for $\ROP{G}$ and $\KOGP$.
% \end{cor}

%+++++++++++++++++++++++++++++++++++++++++++++++++++
\subsection{Multiplicativity of idempotent summands} \label{subsect:intro mult}
We now turn to the multiplicative properties of the idempotent splittings of the complex and real
representation rings and equivariant $K$-theory spectra. Since the block
$ \Loc{\RUP{G}}{\lin(e_C)} $
agrees with the $\AP{G}$-module localization
\[ \Loc{\RUP{G}}{e_C} \cong \RUP{G} \otimes_{\AP{G}} \Loc{\AP{G}}{e_C}, \]
we obtain an identification
\[ \Loc{\KUGP}{\lin(e_C)} \simeq \KUGP \wedge \Loc{\sphere_{(\Primes)}}{e_C} \]
of the blocks of $\Primes$-local $G$-equivariant $K$-theory with an $e_C$-localization in genuine $G$-spectra. By
Corollary~\ref{intro cor subrings}, the same is true for $\ROP{G}$ and $\KOGP$.
This enables us to reduce the solution to Problem~\ref{problem} for equivariant $K$-theory to the one for the
sphere given in the prequel \cite{boehme:mult-idempot}. The resulting classification of the maximal
$N_\infty$ ring structures of the idempotent summands of $\KUGP$ can be summarized as follows:

\begin{thm} \label{intro summary thm}
Let $C \leq G$ be a cyclic group of order not divisible by any prime in $\Primes$ and let $e_C$ be the
corresponding primitive idempotent in $\AP{G}$.
% Write $e_C' = \lin(e_C)$ for the primitive idempotent $e_C' \in \RUP{G}$.
Let $K \leq H \leq G$ be nested subgroups. Then the following are equivalent:
\begin{enumerate}[(a)]
  \item The $G$-spectrum $\Loc{\sphere_{(\Primes)}}{e_C}$ inherits a norm map $\tilde{N}_K^H$ from the norm map
  $N_K^H$ of $\sphere_{(\Primes)}$.
  \item The $G$-spectrum $\Loc{\KUGP}{e_C}$ inherits a norm map $\tilde{N}_K^H$ from
%   the norm map $N_K^H$
  that of $\KUGP$.
  \item The Mackey functor $\Loc{\AP{-}}{e_C}$ inherits a norm map $\tilde{N}_K^H$ from
%   the norm map $N_K^H$
  that of $\AP{-}$.
  \item The Mackey functor $\Loc{\RUP{-}}{e_C}$ inherits a norm map $\tilde{N}_K^H$ from
%   the norm map $N_K^H$
  that of $\RUP{-}$.
  \item Any subgroup $C' \leq H$ conjugate in $G$ to $C$ lies in $K$.
\end{enumerate}
All of the above holds with $\KUGP$ and $\RUP{-}$ replaced by their real variants $\KOGP$ and $\ROP{-}$.
\end{thm}

The equivalence of (a), (c) and (e) was already proven in \cite{boehme:mult-idempot}.
% , up to a rephrasing of statement (e) explained in Lemma~\ref{lemma malte}.
Theorem~\ref{intro summary thm} is made more precise in Theorem~\ref{thm rep splitting norms}, Theorem~\ref{thm rep splitting ITF}
and Corollary~\ref{thm htpy} in terms of Blumberg and Hill's framework of incomplete Tambara functors and
$N_\infty$ operads.

\begin{rem}
If $H$ does not contain a group conjugate in $G$ to $C$, then the norms $\tilde{N}_K^H$ exist for trivial
reasons: It can be seen from Theorem~\ref{Dress idempotents} that the restriction of $e_C$ to $H$ vanishes,
and so $\Loc{\AP{H}}{e_C}$ and $\Loc{\RUP{H}}{e_C}$ must be zero. In other cases, these groups are always
non-zero.
\end{rem}

An immediate consequence of Theorem~\ref{intro summary thm} is the following:

\begin{cor} \label{intro cor principal block KU}
The summand $\Loc{\KUGp}{e_C}$ is a $G$-$E_\infty$ ring spectrum
% (i.e., is equipped with all possible norm maps)
if and only if $C \leq G$ is the trivial group. The same is true for real $K$-theory.
\end{cor}

\subsection{Organization}
In §\ref{section:idpt}, we recall Dress' work on idempotents in the Burnside ring and give a proof of
Theorem~\ref{intro thm idempotents RU}. The algebraic and homotopical parts of Theorem~\ref{intro summary thm}
are discussed in §\ref{section:alg} and §\ref{section:htpy}, respectively.

\subsection{Acknowledgements}
The present work was part of the author's PhD project at the University of Copenhagen;
a previous version of the article was included in his PhD thesis \cite{boehme:thesis}.
The author would like to thank his PhD advisor Jesper Grodal,
his PhD committee consisting of Andrew Blumberg, John Greenlees and Lars Hesselholt,
as well as Markus Hausmann, Joshua Hunt, Malte Leip, Riccardo Pengo, David Sprehn
and an anonymous referee
for many helpful discussions and suggestions.
This research was supported by the Danish National Research Foundation through the Centre
for Symmetry and Deformation (DNRF92).

%Section: Idempotents of RU(G)
\section{Idempotent elements in representation rings} \label{section:idpt}
The goal of this section is to prove Theorem~\ref{intro thm idempotents RU}.
In §\ref{subsect:idpts of A(G)}, we show how some parts of the theorem follow easily from the classification
of idempotents in the Burnside ring. The difficult part is to prove that the images of the Burnside ring
idempotents are indeed primitive. We recall Atiyah's description \cite{atiyah:characters} of the prime ideal
spectrum $\Spec(\RUP{G} \otimes \sint)$ in §\ref{subsect:atiyah}, where $\sint$ is obtained from $\Z$ by
adjoining sufficiently many roots of unity, classify the idempotents of $\RUP{G} \otimes \sint$ in
§\ref{subsect:idempotents over splitting field}, and deduce the primitivity part of Theorem~\ref{intro thm
idempotents RU} in §\ref{subsect:idpt RU}. In the rational and in the $p$-local case, it is possible to prove
the primitivity in an easier way, as we explain in §\ref{subsect:quick proofs}.

% \begin{enumerate}
%   \item[§~\ref{subsect:atiyah}] We work with coefficients in the ring of integers $\sint$ of a number field
%   $\sfield$ containing sufficiently many roots of unity and recall Atiyah's description
%   \cite{atiyah:characters} of the prime ideal spectrum $\Spec(\RUP{G} \otimes \sint)$.
%   \item[§~\ref{subsect:idempotents over splitting field}] We give a group-theoretic classification of the
%   connected components of the space $\Spec(\RUP{G} \otimes \sint)$.
%   \item[§~\ref{subsect:idpt RU}] We obtain $\Spec(\RUP{G})$ from $\Spec(\RUP{G} \otimes \sint)$ as the orbit
%   space under the action of the Galois group of $\sfield$ over $\Q$. The resulting description of the
%   connected components gives the missing step in the proof of Theorem~\ref{intro thm idempotents RU}.
%   \item[§~\ref{subsect:quick proofs}] We give short alternative proofs of the primitivity of the idempotents
%   $\lin(e_C)$ in the rational and in the $p$-local case. The argument does not seem to generalize to the
%   $\Primes$-local setting.
% \end{enumerate}

%++++++++++++++++++++++++++++++++++++++++++++
\subsection{Idempotents in the Burnside ring} \label{subsect:idpts of A(G)}
We recall Dress' classification of idempotents of $\AP{G}$~\cite{dress:solvable}
and prove parts of Theorem~\ref{intro thm idempotents RU}.

\begin{notation}
Recall that by the Chinese remainder theorem applied to $\langle g \rangle$,
each $g \in G$ can be written uniquely as a product $\Ppart{g} \cdot \Pprime{g}$ of powers of $g$,
where $\Ppart{g}$ is of order divisible only by primes in $\Primes$,
and $\Pprime{g}$ is of order prime to $\Primes$.
The elements $\Ppart{g}$ and $\Pprime{g}$ are called
the $\Primes$-\emph{part} and $\Primes$-\emph{prime part} of~$g$, respectively.
\end{notation}

\begin{definition}
The $\Primes$-\emph{residual subgroup} of a group $H$ is the unique minimal normal subgroup $\OP{H}$
such that the quotient is a solvable $\Primes$-group,
i.e.~a solvable group of order only divisible by primes in $\Primes$.
$H$ is called $\Primes$-\emph{perfect} if $\OP{H} = H$.
\end{definition}

\begin{lemma} \label{OP of cyclic group}
For cyclic groups, we have $\OP{g} = \langle \Pprime{g} \rangle$.
In particular, a cyclic group is $\Primes$-perfect
if and only if its order is not divisible by any element of $\Primes$.
\qed
\end{lemma}

\begin{definition}
For a subgroup $H \leq G$, the \emph{mark homomorphism} $\phi^H \colon \AP{G} \to \ZP$ is extended
additively from the assignment $X \mapsto |X^H|$ for finite $G$-sets $X$.
\end{definition}

\begin{thm}[\cite{dress:solvable}, Prop.~2] \label{Dress idempotents}
There is a canonical bijection between the conjugacy classes of $\Primes$-perfect subgroups $L \leq G$ and the set of primitive
idempotent elements of $\AP{G}$. It sends $L$ to the element $e_L \in \AP{G}$ whose marks $\phi^H(e_L)$ at a
subgroup $H \leq G$ are one if $\OP{H}$ and $L$ are conjugate in $G$, and zero otherwise.
\end{thm}

% Now the symbol $e_C$ can refer to an element of $\AP{G}$ or an element of $\RUP{G}$. This slight abuse of
% notation is justified by Theorem~\ref{intro thm idempotents RU}.

% \begin{notation}
Write $\chi(V)(g)$ for the value of the character of $V \in \RUP{G}$ at the element $g \in G$.
% \end{notation}
%
The linearization map $\lin \colon \AP{G} \to \RUP{G}$ satisfies the following simple identity:

\begin{lemma} \label{lemma characters marks}
For $X \in \AP{G}$, we have $\chi(\lin(X))(g) = \phi^{\langle g \rangle}(X) \in \Z$. \qed
\end{lemma}

\begin{cor} \label{cor characters of burnside images}
Let $L \leq G$ be a $\Primes$-perfect subgroup. Then the virtual representation $\lin(e_L)$ has character
values
\[ \chi(\lin(e_L))(g) = \phi^{\langle g \rangle}(e_L) = \begin{cases}
1 & \text{if } \langle \Pprime{g} \rangle \sim_G L \\
0 & \text{otherwise}.
\end{cases} \]
In particular, $\lin(e_L)$ is zero if $L$ is not cyclic. The elements $\lin(e_C)$ are mutually
orthogonal idempotents summing to one, where $C$ ranges over a set of representatives for the conjugacy classes of cyclic
$P$-perfect subgroups.
\end{cor}

\begin{proof}
The statement follows from Theorem~\ref{Dress idempotents}, Lemma~\ref{lemma characters marks},
and Lemma~\ref{OP of cyclic group}.
% By Lemma~\ref{lemma characters marks}, we have
% %
% \[ \chi(\lin(e_L))(g) = \phi^{\langle g \rangle}(e_L) = \begin{cases}
% 1 & \text{if } \OP{\langle g \rangle} \sim_G L \\
% 0 & \text{otherwise}.
% \end{cases} \]
% %
% But $\OP{\langle g \rangle} = \langle \Pprime{g} \rangle$ is conjugate in $G$ to $L$ if and only if $L$ is cyclic with generator $x$
% conjugate in $G$ to $\Pprime{g}$. It now follows from Lemma~\ref{lemma gamma permutes generators} that $\lin(e_L)$ and
% $e_{x.\Gamma}$ have the same character.
\end{proof}

% It follows that the elements $e_C \in \RUP{G}$, where $C$ runs over conjugacy classes of cyclic
% $\Primes$-perfect subgroups $C \leq G$, form a set of mutually orthogonal idempotents that sum to one.
This proves all the statements of Theorem~\ref{intro thm idempotents RU} except for the primitivity of the
idempotents $\lin(e_C)$. Note that the rational case ($\Primes = \emptyset$) of
Corollary~\ref{cor characters of burnside images} is stated in \cite[Theorem]{gluck:bound} and goes back to a
similar result by Solomon \cite[Thm.~3]{solomon:burnside}.

% \begin{rem}
% Note that in the rational case ($\Primes = \emptyset$), the number of distinct elements $e_C' = \lin(e_C)$ equals the rank of $\RU{G}
% \otimes \Q$, hence they form a basis and must be primitive. This is precisely Solomon's original argument
% \cite{solomon:burnside} for the primitivity in the rational setting. For $\Primes \neq \emptyset$, this
% counting argument fails. \redfont{TODO: This is wrong! Solomon treats the ring $RQ(G) \otimes \Q$. Move this to the ``quick proof''
% subsection and fix it, using the lemmas for the quick $p$-local proof!}
% \end{rem}

The following observation is not part of the proof of Theorem~\ref{intro thm idempotents RU}, but we record it
for later reference.

\begin{lemma} \label{lemma kernel lin}
The $\Primes$-local Burnside ring splits as
\[ \AP{G} \cong e_{cyc} \cdot \AP{G} \times e_{ker} \cdot \AP{G} \]
where $e_{cyc}$ (respectively $e_{ker}$) is defined to be the sum of all primitive idempotents $e_L$ with
$L$ cyclic (respectively non-cyclic). Moreover, the summand $e_{ker} \cdot \AP{G}$ is precisely the kernel of
the linearization map $\lin \colon \AP{G} \to \RXP{G}$.
\end{lemma}

\begin{proof}
The first part follows from Theorem~\ref{Dress idempotents} by writing $1 = e_{cyc} + e_{ker}$.
Lemma~\ref{lemma characters marks} implies that the kernel of $\lin$ consists of those virtual $G$-sets whose
marks vanish at all cyclic subgroups. By Corollary~\ref{cor characters of burnside images}, these are
precisely the elements of the ideal $e_{ker} \cdot \AP{G}$.
\end{proof}

\subsection{Prime ideals in the splitting field case} \label{subsect:atiyah}
Let $\exp(G)$ be the exponent\footnote{The \emph{exponent} of a finite group is the least common multiple of
the orders of all group elements.} of $G$ and write $\sfield$ for the $\exp(G)$-th cyclotomic extension of
$\Q$ with ring of integers $\sint$ and Galois group $\Gamma := Gal(\sfield \colon \Q)$. All characters of $G$-representations over
the complex numbers take values in $\sint$, and therefore can be viewed as class functions $\CC{G} \to \sint$, where $\CC{G}$ is the
set of conjugacy classes of $G$.
When working $\Primes$-locally, the elements of $\RUP{G}$ are $\ZP$-linear combinations of
irreducible representations of $G$ over the complex numbers, hence their characters take values in $\sintP :=
\sint \otimes \ZP$.

\begin{notation} \label{notation character}
Any element $V \in \RUP{G} \otimes \sint$ can be written as an $\sintP$-linear combination $V = \sum_i
\lambda_i \cdot V_i$ of irreducible $G$-representations $V_i$. We write
\[ \hat{\chi}(V)(g) := \sum_i \lambda_i \cdot \chi(V_i)(g) \]
for the value of the $\sintP$-linear character of $V$ at $g \in G$. 
% We write $\chi_g \colon \RU{G} \otimes \sint \to \sint$ for the map $V \mapsto \chi(V)(g)$ that evaluates a
% character at the group element $g \in G$.
\end{notation}

Characters are multiplicative in tensor products of representations,
hence $\hat{\chi}(-)(g)$ defines a ring homomorphism $\RUP{G} \otimes \sint \to \sintP$.
Atiyah \cite{atiyah:characters} described the structure of the prime ideal spectrum $\Spec(\RU{G} \otimes \sint)$
in terms of these maps.
His proof applies without changes to the open subscheme $\Spec(\RUP{G} \otimes \sint)$
cut out by $\Primes$-localization.

\begin{prop}[Cf.~\cite{atiyah:characters}, Prop.~6.4] \label{prime ideals RUP_G}
The topological space $\Spec(\RUP{G} \otimes \sint)$ can be described as follows:
\begin{enumerate}[(1)]
  \item Every prime ideal of $\RUP{G} \otimes \sint$ is of the form
  \[ \Qpg{\fp}{g} := (\hat{\chi}(-)(g))^{-1}(\fp)
  = \{ V \in \RUP{G} \otimes \sint \, | \, \hat{\chi}(V)(g) \in \fp \} \]
%   = \{ \sum V \otimes \lambda \, | \, \sum \chi(V) \cdot \lambda \in \fp \} \]
  %
  for some element $g \in G$ and some prime ideal $\fp \trianglelefteq \sintP$.
  \item Let $\fp, \fq \trianglelefteq \sintP$ be prime ideals such that $\Z \cap \fq = q\Z$ for a prime $q \in \Z$.
  There is an inclusion $\Qpg{\fp}{g} \subseteq \Qpg{\fq}{h}$
  if and only if $\fp$ is contained in $\fq$
  and $\qprime{g}$ is conjugate in $G$ \linebreak to $\qprime{h}$.
  \item The prime ideals $\Qpg{\fp}{g}$ with $\fp = (0)$ are minimal and the ones with $\fp \neq (0)$ are maximal. In particular,
  the Krull dimension of $\RUP{G} \otimes \sint$ is one.
\end{enumerate}
\end{prop}

\subsection{Idempotents in the splitting field case} \label{subsect:idempotents over splitting field}

We can deduce a classification of the idempotent elements of $\RUP{G} \otimes \sint$ from Proposition~\ref{prime
ideals RUP_G}.
% Since we were unable to find it in the literature, we give a detailed proof that is similar to
Our proof is inspired by
Dress' approach \cite[Prop.~2]{dress:solvable} to the idempotents in the Burnside ring.

\begin{thm} \label{components of RUP(G)}
The map
\[ G \to \pi_0(\Spec(\RUP{G} \otimes \sint)) \]
that sends $x \in G$ to the connected component of $\Qpg{0}{x}$ induces a bijection between
the set of conjugacy classes of $\Primes$-prime elements\footnote{A $\Primes$-\emph{prime} element is an
element $x$ such that $\Pprime{x} = x$.} of $G$ and the set of connected components of $\Spec(\RUP{G} \otimes
\sint)$.
In particular, the prime ideal spectrum of $\RU{G} \otimes \sint$ is connected.
\end{thm}

This follows directly from:

\begin{prop}
For any (not necessarily $\Primes$-prime) elements $x, y \in G$, the prime ideals $\Qpg{\fp}{x}$ and
$\Qpg{\fq}{y}$ lie in the same connected component of $\Spec(\RU{G} \otimes \sint)$ if and only if
$\Pprime{x}$ and $\Pprime{y}$ are conjugate in $G$.
\end{prop}

\begin{proof}
First observe that for $\fp \neq (0)$, the height one ideal $\Qpg{\fp}{x}$ lies in the closure of the height
zero ideal $\Qpg{(0)}{x}$, so without loss of generality we may assume that $\fp = \fq = (0)$.
Since $\RUP{G} \otimes \sint$ has Krull dimension one, two points $\Qpg{(0)}{x}$ and $\Qpg{(0)}{y}$ lie in the
same component if and only if there is a zig-zag of inclusions of prime ideals \\
\begin{center}
$ \xymatrix@M=5pt@R=1.3pc{
	\; & \Qpg{\fp_0}{x_0} = \Qpg{\fp_0}{x_1} & \; & \ldots & \; \\
	\Qpg{(0)}{x_0} \ar[ur] & \; & \Qpg{(0)}{x_1} \ar[ul] \ar[ur] & \; & \Qpg{(0)}{x_r} \ar[ul]
}$
\end{center}
for some elements $x = x_0, x_1, \ldots, x_r = y \in G$
and some prime ideals $\fp_i \trianglelefteq \sintP$.
By part (2) of Theorem~\ref{prime ideals RUP_G},
we have an equality $\Qpg{\fp_i}{x_i} = \Qpg{\fp_i}{x_{i+1}}$
if and only $\argprime{(x_i)}{p_i}$ is conjugate in $G$ to $\argprime{(x_{i+1})}{p_i}$,
where $p_i$ is given by $\Z \cap \fp_i = p_i \Z$. \\
For the ``only if'' part of the proposition, given a zig-zag as above, it follows that
\[ \Pprime{x} = \Pprime{(\argprime{(x_0)}{p_0})}
\sim_G \Pprime{(\argprime{(x_1)}{p_0})}
= \Pprime{(\argprime{(x_1)}{p_1})}
\sim_G \ldots
\sim_G \Pprime{(\argprime{(x_r)}{p_{r-1}})}
= \Pprime{y} \]
where $\sim_G$ indicates being conjugate in $G$. \\
For the ``if'' part, assume that $\Pprime{x} \sim_G \Pprime{y}$. Since the prime ideals $\Qpg{(0)}{g}$ only
depend on the conjucagy class of $g$, it follows that $\Qpg{(0)}{\Pprime{x}}$ and $\Qpg{(0)}{\Pprime{y}}$
agree. Thus, it suffices to show that for any $g \in G$, the prime ideals $\Qpg{(0)}{g}$ and
$\Qpg{(0)}{\Pprime{g}}$ lie in the same component.
% We will do this by constructing an explicit zig-zag as above.
We will construct an explicit zig-zag as above.
Let $p_0, p_1, \ldots, p_r$ be all primes in $\Primes$ that divide the order of $G$. By the going-up theorem,
we can find prime ideals $\fp_i \trianglelefteq \sintP$ such that $\fp_i \cap \Z = p_i \Z$.
Then $\Pprime{g}$ may be computed as
\[ \Pprime{g} = \argprime{(\cdots \argprime{(\argprime{g}{p_0})}{p_1} \cdots)}{p_r}. \]
Inductively, define $g_0 := g$ and $g_i := \argprime{(g_{i-1})}{p_{i-1}}$
so that we have $\argprime{(g_i)}{p_i} = \argprime{(g_{i+1})}{p_i}$.
Then these choices of elements $g_i$ and prime ideals $\fp_i$ give rise to a zig-zag
between $\Qpg{(0)}{g}$ and $\Qpg{(0)}{\Pprime{g}}$,
which completes the proof.
\end{proof}

% Recall that an idempotent element of a commutative ring is \emph{primitive} if it cannot be written as a sum of idempotents all of
% which are non-zero. \redfont{(TODO: Mention earlier? Introduction?)}

\begin{cor} \label{idempotents of RUP(G)}
The conjugacy classes of $\Primes$-prime elements $(x)$ of $G$ are in canonical bijection with the primitive idempotents $e_x$ of
$\RUP{G} \otimes \sint$.
The character of the element $e_x$ is given as follows:
\[ \hat{\chi}(e_x)(g) = \begin{cases}
1 & \text{if } \Pprime{g} \sim_G x \\
0 & \text{otherwise}
\end{cases} \]
\end{cor}

\begin{proof}
It is a standard fact of algebraic geometry that for any commutative ring $R$, the subsets $V \subseteq \Spec(R)$ that are both open
and closed are in canonical bijection with the idempotent elements of $R$, by assigning to $V$ the global section\footnote{Here we
use that by definition, the global sections of $\Spec(R)$ agree with the ring $R$.} which is constant one on $V$ and constant zero
on the complement of $V$.
Under this identification, the primitive idempotents correspond to the minimal non-empty open and closed subsets. The latter agree
with the connected components of $\Spec(\RUP{G} \otimes \sint)$ since there are only finitely many of them. The first claim now
follows from Theorem~\ref{components of RUP(G)}. For the description of characters, note that $\hat{\chi}(e_x)(g) = 1$ if and only if the
corresponding global section $e_x$ evaluates to one at the point $\Qpg{(0)}{g}$ if and only if $\Qpg{(0)}{x}$ and $\Qpg{(0)}{g}$
are in the same connected component.
\end{proof}

\begin{rem}
Roquette~\cite{roquette:untersuchung} shows that the classification given in Corollary~\ref{idempotents of
RUP(G)} also holds for the primitive idempotents in the $p$-adic representation ring after adjoining all
$e$-th roots of unity.
% representation algebra $\RU{G} \otimes \Z_p^\wedge \left[ \sqrt[e]{1} \right]$.
\end{rem}

\begin{rem}
Using Schur's orthogonality relations, it follows from Corollary~\ref{idempotents of RUP(G)} that $e_x$ is given explicitly as
\[ e_x = \frac{1}{|C_G(x)|} \sum_{V} \chi(V)(x^{-1}) \cdot V \]
where $V$ runs over a system of representatives of the irreducible representations of $G$ and $C_G(x)$ denotes the centralizer of
$x$ in $G$. This observation goes back at least to Brauer \cite[(7)]{brauer:artin}.
The coefficients can also be expressed in terms of M\"obius functions, see \cite[Thm.~4]{solomon:burnside},
\cite[Prop.]{gluck:idempotent} and \cite[§3]{yoshida:idempotents}.
\end{rem}

%++++++++++++++++++++++++++++++++++++++++++++++++++++++++++++
\subsection{Idempotents of $\RUP{G}$} \label{subsect:idpt RU}
Recall that $\Gamma \cong (\Z/\exp(G))^\times$ denotes the Galois group of the cyclotomic extension
$\sfield/\Q$.
The left $\Gamma$-action on $\sfield$ restricts to an action on $\sint$. Let $\Gamma$ act on $\RUP{G} \otimes \sint$ via its action
on the right factor. Then clearly $\RUP{G} = (\RUP{G} \otimes \sint)^\Gamma$. The group $\Gamma$ then acts from the right on
$\Spec(\RUP{G} \otimes \sint)$ and we have $\Spec(\RUP{G}) \cong (\Spec(\RUP{G} \otimes \sint))/\Gamma$,
cf.~\cite[§11.4, Exerc.~11.4]{serre:linear-reps}.
% Thus all prime ideals of $\RUP{G}$ are $\Gamma$-orbits of those of $\RUP{G} \otimes \sint$.
We will now describe these $\Gamma$-orbits in terms of the prime ideals $\Qpg{\fp}{x}$.

First recall that the left $\Gamma$-action on $\sint$ induces a right $\Gamma$-action on $\Spec(\sint)$ that is given by
$\fp.\gamma = \gamma^{-1}(\fp)$.

\begin{definition}[\cite{serre:linear-reps}, §12.4]
Define a right $\Gamma$-action on the underlying set of $G$ as follows: If $\gamma \in \Gamma$ corresponds to
the unit $m \in (\Z/\exp(G))^\times$, let $g.\gamma := g^{m^{-1}}$, where $m^{-1}$ is (any integer
representing) the inverse of $m$ in the group $\Gamma$.
% Then $\Gamma$ acts from the right on the set $G$ via $g.\gamma := \gamma^{-1}.g$.
\end{definition}

This action is well-defined since the order of $g \in G$ divides $\exp(G)$. Moreover, it is compatible with
conjugation in $G$. We can describe the $\Gamma$-orbits in $G$ easily:

\begin{lemma}[\cite{serre:linear-reps}, §13.1, Cor.] \label{lemma gamma permutes generators}
Two elements $x,y \in G$ lie in the same $\Gamma$-orbit if and only if they generate the same cyclic subgroup of $G$.
\end{lemma}

\begin{proof}
Let $n$ divide $\exp(G)$. Then $\Gamma \cong (\Z/\exp(G))^\times$ permutes the
generators of $\Z/\exp(G)$ transitively, and the same is true for the generators of the cyclic group
$\Z/n$, viewed as a subgroup of $\Z/\exp(G)$.
\end{proof}

\begin{prop}
The left $\Gamma$-action on $\RUP{G} \otimes \sint$ induces a right $\Gamma$-action on the space
$\Spec(\RUP{G} \otimes \sint)$ which coincides with the action defined by $\Qpg{\fp}{x}.\gamma =
\Qpg{\fp.\gamma}{x.\gamma}$.
\end{prop}

\begin{proof}
As in Notation~\ref{notation character}, write $\sum_i \lambda_i \cdot V_i$ for a generic element of $\RUP{G}
\otimes \sint$. Then
\begin{align*}
\Qpg{\fp}{x}.\gamma = \gamma^{-1}(\Qpg{\fp}{g})
% = \{ \sum \lambda \cdot V \, | \, \sum \gamma(\lambda) \cdot V \in \Qpg{\fp}{g} \}
&= \{ \sum_i \lambda_i \cdot V_i \, | \, \sum_i \gamma(\lambda_i) \cdot \chi(V_i)(g) \in \fp \} \\
&= \{ \sum_i \lambda_i \cdot V_i \, | \, \sum_i \lambda_i \cdot \gamma^{-1}(\chi(V_i)(g)) \in \fp.\gamma \} \\
&= \{ \sum_i \lambda_i \cdot V_i \, | \, \sum_i \lambda_i \cdot \chi(V_i)(g.\gamma) \in \fp.\gamma \} \\
&= \Qpg{\fp.\gamma}{g.\gamma} \qedhere
\end{align*}
\end{proof}

\begin{cor} \label{cor components of RUP(G)}
The map
\[ G \to \pi_0(\Spec(\RUP{G})) \]
that sends an element $x$ to the component of the orbit $\Qpg{(0)}{x}.\Gamma$ induces a bijection
between the $\Gamma$-orbits of conjugacy classes of $\Primes$-prime elements $x \in G$ and the set of components of the
topological space $\Spec(\RUP{G})$. In particular, the spectrum of $\RU{G}$ is connected.
\end{cor}

\begin{cor} \label{cor classification in terms of Gamma-orbits}
There is a canonical bijection between the $\Gamma$-orbits of conjugacy classes of $\Primes$-prime elements $x \in G$
and the primitive idempotents in $\RUP{G}$. The idempotent $e_{x.\Gamma}$ associated to the orbit of $(x)$ has character given by
\[ \chi(e_{x.\Gamma})(g) = \begin{cases}
1 & \text{if } \Pprime{g} \sim_G x.\gamma \text{ for some } \gamma \in \Gamma \\
0 & \text{otherwise}
\end{cases} \]
\end{cor}

\begin{proof}
This follows from Corollary~\ref{cor components of RUP(G)} in the same way that Corollary~\ref{idempotents of RUP(G)} follows from
Theorem~\ref{components of RUP(G)}, see the proof of Corollary~\ref{idempotents of RUP(G)}.
\end{proof}

\begin{rem} \label{rem idempotents of RU as sums of splitting idempotents}
In particular, we have $e_{x.\Gamma} = \sum_{\gamma \in \Gamma} e_{x.\gamma}$ in $\RUP{G}$.
A simple calculation shows that $e_{x.\gamma} = \gamma^{-1}(e_x)$ in $\RUP{G} \otimes \sint$. Therefore $e_{x.\Gamma} =
tr_{\sfield/\Q}(e_x)$ is the field trace of $e_x$. We will not use this fact.
\end{rem}

By Lemma~\ref{lemma gamma permutes generators}, we can write $e_{\langle x \rangle}' := e_{x.\Gamma}$ and
rephrase Corollary~\ref{cor classification in terms of Gamma-orbits} in terms of cyclic subgroups. At
this point, there is no dependence on the field extension $\sfield/\Q$ anymore.

\begin{cor} \label{cor classification in terms of cyclic subgroups}
There is a canonical bijection between the conjugacy classes of cyclic $\Primes$-perfect subgroups $C \in G$
and the primitive idempotents in $\RUP{G}$. The primitive idempotent $e_C'$ has character given by
\[ \chi(e_C')(g) = \begin{cases}
1 & \text{if } \langle \Pprime{g} \rangle \sim_G C \\
0 & \text{otherwise.}
\end{cases} \]
In particular, the character of $e_C'$ agrees with that of $\lin(e_C)$ given in Corollary~\ref{cor characters
of burnside images} and hence we have $e_C' = \lin(e_C)$.
\end{cor}

Theorem~\ref{intro thm idempotents RU} follows.

\begin{rem} \label{rem relevant primes}
It is clear from Corollary~\ref{cor classification in terms of cyclic subgroups} that the primitive
idempotents of $\RUP{G}$ only depend on those primes $p \in \Primes$ that divide the order of $G$.
\end{rem}

% \begin{proof}[Proof of Theorem~\ref{intro thm idempotents RU}]
% By Corollary~\ref{cor classification in terms of cyclic subgroups}, the elements $e_C$ are precisely the
% primitive idempotents of $\RUP{G}$. By Corollary~\ref{cor characters of burnside images}, they coincide with
% the images $\lin(e_C)$ of primitive idempotents of the Burnside ring parametrized by cyclic subgroups.
% \end{proof}

% \begin{proof}[Proof of Corollary~\ref{intro cor subrings}]
% The map $\lin \colon \AP{G} \to \RUP{G}$ factors through $\ROP{G}$, hence all primitive idempotents of $\RUP{G}$ lie in the subring
% $\ROP{G}$. They sum to one and must be primitive in $\ROP{G}$ because they are in $\RUP{G}$. The same argument works with $\R$
% replaced by any other subfield of $\C$.
% \end{proof}

%+++++++++++++++++++++++++++++++++++++++++++++++++++
\subsection{Quick proofs of special cases} \label{subsect:quick proofs}
In the rational and $p$-local case, we can give short ad-hoc proofs of the primitivity of the elements $\lin(e_C)$ stated as
part of Theorem~\ref{intro thm idempotents RU}.

\begin{lemma} \label{lemma constant on cyclic generators}
Let $x, y \in G$ generate the same subgroup. If all character values of the virtual representation $V \in
\RUP{G}$ lie in $\ZP$, then
$ \chi(V)(x) = \chi(V)(y). $ 
\end{lemma}

\begin{proof}
By Lemma~\ref{lemma gamma permutes generators}, we can find $\gamma \in \Gamma$ such that $y = x.\gamma$. Then
\[ \chi(V)(y) = \chi(V)(x.\gamma) = \gamma^{-1}(\chi(V)(x)) = \chi(V)(x) \]
because $\chi(V)(x) \in \ZP = (\sintP)^\Gamma$.
\end{proof}

\begin{cor} \label{cor rational primitivity}
For any cyclic $C \leq G$, the idempotent $\lin(e_C) \in \RU{G} \otimes \Q$ is primitive.
\end{cor}

\begin{proof}
Recall that the character of $\lin(e_C)$ is one on elements that generate subgroups conjugate to $C$ and zero otherwise. But
Lemma~\ref{lemma constant on cyclic generators} shows that any integer-valued character must be constant on the set where
$\lin(e_C)$ is one, hence $\lin(e_C)$ cannot decompose as a sum of idempotents.
\end{proof}

For the $p$-local case, we need another lemma. It was used in Atiyah's proof of Theorem~\ref{prime ideals RUP_G}. 

\begin{lemma}[\cite{atiyah:characters}, proof of Lemma~6.3] \label{lemma atiyah}
Let $V \in \RUp{G}$ and let $\fp$ be a prime of $\sintp = \sint \otimes \Zp$. Then
$ \chi(V)(g) \equiv \chi(V)(\pprime{g}) \mod \fp. $
\end{lemma}

\begin{proof}
Without loss of generality, we may assume that $G$ is cyclic and $V$ one-dimensional, hence its character is
multiplicative.
Write $g = \pprime{g} \cdot h$ where the order of $h$ is $p^r$, then $ ( \chi(V)(h) )^{p^r} = 1 $. But
$\sintP/\fp$ is a finite field of characteristic $p$, so
\[ \chi(V)(h) \equiv 1 \mod \fp, \]
and consequently
\[ \chi(V)(g) \equiv \chi(V)(\pprime{g}) \cdot \chi(V)(h) \equiv \chi(V)(\pprime{g}) \mod \fp. \qedhere \]
%
% \redfont{(JG: fix)}
% %
% \[ \chi(V)(g) \equiv ( \chi(V)(g) )^{p^r} \equiv
% ( \chi(V)(\pprime{g}) )^{p^r} \cdot ( \chi(V)(h) )^{p^r} \equiv \chi(V)(\pprime{g}) \mod \fp. \]
% %
\end{proof}

\begin{definition}
For $C \leq G$ cyclic of order prime to $p$, let
\[ S_C := \{ g \in G \, | \, \langle \pprime{g} \rangle \sim_G C \}. \]
\end{definition}

Combining Lemma~\ref{lemma constant on cyclic generators} and Lemma~\ref{lemma atiyah} gives: 

\begin{cor}
Let $\fp$ be any prime ideal in $\sintp$. If all character values of $V \in \RUp{G}$ lie in $\Zp$, then the
character of $V$ is constant modulo $\fp$ on the set $S_C$.
\end{cor}

Finally, a proof similar to that of Corollary~\ref{cor rational primitivity} shows:

\begin{cor}
For any cyclic $p$-perfect $C \leq G$, the idempotent $\lin(e_C) \in \RUp{G}$ is
primitive.
\end{cor}

% \begin{proof}
% The characters of an idempotent can only attain the values zero or one, so they must be constant on the set
% $S_C$.
% \end{proof}

Lemma~\ref{lemma atiyah} does not hold in the general $\Primes$-local case, as the next example shows.
However, it follows from Theorem~\ref{intro thm idempotents RU} that the statement becomes true
under the additional assumption that the character of $V$ be zero outside of $S_C$. We do not know how to use
this assumption to give a quick proof of the primitivity of the elements $\lin(e_C)$ that applies to all choices of
$\Primes$.

\begin{example}
Let $G = C_2 \times C_3$ be the cyclic group of order $6$ and $\Primes = \{ 2, 3 \}$. Write $\one$ for the
trivial representation and let $V \in \RUP{G}$ be given as the tensor product of the sign representation of
$C_2$ with the sum of the two non-trivial irreducible $C_3$-representations. Let $g \in G$ be a generator and
observe that $\Pprime{g} = 1$. However,
\[ \chi(V-\one)(g) = 0 \not\equiv 1 = \chi(V-\one)(1) \mod \fp \]
for any prime ideal $\fp$ of $\sintP$.
\end{example}

%Section: Algebraic results
\section{Idempotent splittings of representation rings} \label{section:alg}
As before, let $\Primes$ be a fixed collection of prime numbers, and let $\RXP{G}$ denote one of the rings
$\ROP{G}$ or $\RUP{G}$. The goal of this section is to describe the multiplicativity of the idempotent
splitting
\[ \RXP{-} \cong \prod_{(C)} \Loc{\RXP{-}}{e_C}. \]
We start by briefly recalling the notion of an (incomplete) Tambara functor in §\ref{subsect:ITFs}. In
§\ref{subsect:alg mult summands}, we study the multiplicativity of the idempotent splitting of $\RXP{-}$:  we
characterize the norms which are compatible with $e_C$-localization in Theorem~\ref{thm rep splitting norms}
and describe the incomplete Tambara functor structure of each idempotent summand in Theorem~\ref{thm rep
splitting ITF}. It is then easy to read off the structure that is preserved by the entire splitting, as we
explain in §\ref{subsect:alg mult splitting}.

% We will determine the maximal incomplete Tambara functor structures available on each of the summands
% $\Loc{\RXP{-}}{e_C}$ in §\ref{subsect:alg mult summands}, and the maximal Tambara functor structure preserved by the product decomposition in
% §\ref{subsect:alg mult splitting}. For convenience, we briefly recall the notion of an (incomplete) Tambara
% functor in §\ref{subsect:ITFs}.

%+++++++++++++++++++++++++++++++++++++++
\subsection{Incomplete Tambara functors} \label{subsect:ITFs}
Recall that many naturally arising Mackey functors have additional multiplicative structure.

\begin{definition}
A \emph{Green functor} is a Mackey functor $\uR$ equipped with commutative ring structures on the values $\uR(H)$ for all $H \leq G$
such that all restrictions maps $R^H_K \colon \uR(H) \to \uR(K)$ become ring homomorphisms and all transfer maps $T_K^H \colon
\uR(K) \to \uR(H)$ are morphisms of $\uR(H)$-modules.
\end{definition}

Often, Green functors come equipped with additional \emph{multiplicative transfer maps} or \emph{norms}
$N_K^H \colon \uR(K) \to \uR(H)$ for all subgroup inclusions $K \leq H \leq G$, satisfying a number of compatibility relations for
norms, additive transfers and restrictions. Tambara \cite{tambara} axiomatized the structure of these objects and called them
\emph{TNR-functors}; nowadays they are referred to as \emph{Tambara functors}.
% They can be alternatively be defined as certain presheaves on a category of bispans of $G$-sets, see \redfont{REF}.

Blumberg and Hill \cite{BH:ITF} introduced the more general notion of an \emph{(incomplete)} $\cI$-\emph{Tambara functor} that
only admits a partial collection of norms for certain subgroup inclusions $K \leq H \leq G$, parametrized by well-behaved
collections $\cI$ of \emph{admissible} $H$-\emph{sets} $H/K$. These \emph{indexing systems} form a poset under inclusion. Thus,
$\cI$-Tambara functors for varying $\cI$ interpolate between the notion of a Green functor (which doesn't necessarily admit any
norms) and that of a Tambara functor (which admits all possible norms). We refer to the above sources for precise definitions and
further details.

\begin{example} \label{example TFs}
The Mackey functors defined by the Burnside ring $\A{-}$ and the representation rings $\RX{-}$ are examples of
Tambara functors. The multiplicative norms $N_K^H$ in the Burnside ring are induced by the co-induction
functor $map_K(H,-)$ from finite $K$-sets to finite $H$-sets. Those of the representation ring are induced by
tensor induction of representations. The linearization maps $\lin \colon \A{-} \to \RX{-}$ are maps of Tambara
functors. \\
More generally, Brun \cite{brun:eqvar-spectra} showed that the zeroth equivariant homotopy groups of a $G$-$E_\infty$ ring spectrum
naturally form a Tambara functor.
\end{example}

\begin{rem}
Note that the $\RO{G}$-graded homotopy groups of a $G$-$E_\infty$ ring spectrum form a \emph{graded Tambara functor}
in the sense of Angeltveit-Bohmann \cite{angeltveit-bohmann}.
% In this article, we will never consider homotopy groups away from degree 0.
\end{rem}

%++++++++++++++++++++++++++++++++++++++++++++++
\subsection{Multiplicativity of the idempotent summands} \label{subsect:alg mult summands}
Observe that the canonical localization maps $\RXP{-} \to \Loc{\RXP{-}}{e_C}$ are levelwise ring homomorphisms that are
compatible with the Mackey functor structure, hence the idempotent splitting of $\RXP{G}$ induces a splitting of the underlying
Green functor of $\RXP{-}$. 
Our next goal is to describe the idempotent summands $\Loc{\RXP{-}}{e_C}$ by proving the equivalence of the
statements (c), (d) and (e) of Theorem~\ref{intro summary thm}. For convenience of the reader, we record this
in the following theorem.

\begin{thm} \label{thm rep splitting norms}
Let $C \leq G$ be a cyclic $\Primes$-perfect subgroup and let $e_C \in \AP{G}$ be the corresponding primitive
idempotent element. Fix subgroups $K \leq H \leq G$. Then the following are equivalent:
\begin{enumerate}
  \item[(c)] The norm map $N_K^H \colon \AP{K} \to \AP{H}$ descends to a well-defined map of multiplicative
  monoids
  \[ \tilde{N}_K^H \colon \Loc{\AP{K}}{e_C} \to \Loc{\AP{H}}{e_C}. \]
  \item[(d)] The norm map $N_K^H \colon \RXP{K} \to \RXP{H}$ descends to a well-defined map of multiplicative
  monoids
  \[ \tilde{N}_K^H \colon \Loc{\RXP{K}}{e_C} \to \Loc{\RXP{H}}{e_C}. \]
  \item[(e)] Any subgroup $C' \leq H$ conjugate in $G$ to $C$ lies in $K$.
\end{enumerate}
\end{thm}

% In the prequel \cite{boehme:norms-sphere}, statement (e) appeared in a slightly more complicated form as the
% statement (e') of the next lemma. The author is grateful to Malte Leip for pointing out this simplification.
% 
% \begin{lemma} \label{lemma malte}
% The statement (e) is equivalent to
% \begin{enumerate}
%   \item[(e')] One of the following holds:
% 	\begin{enumerate}[i)]
%   		\item Neither $K$ nor $H$ are super-conjugate in $G$ to $L$.
%   		\item Both $K$ and $H$ are super-conjugate in $G$ to $L$ and satisfy the following: If $L' \leq G$ is
%   		conjugate in $G$ to $L$ and is contained in $H$, then it is contained in $K$. \qed
% 	\end{enumerate}
% \end{enumerate} 
% \end{lemma}

It was proven in \cite[Thm.~4.1]{boehme:mult-idempot} that the statements (c) and (e) are equivalent, so
Theorem~\ref{thm rep splitting norms} reduces to showing that (c) and (d) are equivalent.

% \begin{prop} \label{prop c iff d}
% The statements (c) and (d) of Theorem~\ref{intro summary thm} are equivalent. 
% \end{prop}

We recall the following fact due to Blumberg and Hill (see \cite[Thm.~2.33]{boehme:mult-idempot} for an
elementary proof):

\begin{thm}[\cite{BH:ITF}, Thm.~5.25] \label{thm alg preservation}
Let $\uR$ be an $\cI$-Tambara functor structured by an indexing system $\cI$.
Let $x \in \uR(G)$. Then the orbit-wise localization $\Loc{\uR}{x}$ is a localization in the category of $\cI$-Tambara functors if
and only if for all admissible sets $H/K$ of $\cI$, the element $N^H_K R^G_K(x)$ divides a power
of $R^G_H(x)$.
\end{thm}

If the element $x$ is idempotent, then checking the above division relation amounts to checking an equation:

\begin{lemma} \label{lemma little trick}
Let $e, e' \in R$ be idempotents in a commutative ring. Then $e$ divides $e'$ if and only if $e \cdot e' =
e'$.
\end{lemma}

\begin{proof}
Assume that $e$ divides $e'$. Then $e' \in eR$, hence $e \cdot e' = e'$, since multiplication by $e$ is
projection onto the idempotent summand $eR$ of $R$. The other direction is obvious.
\end{proof}

% Before we can prove Prop.~\ref{prop c iff d}
% We record the following observation for later reference.
% 
% \begin{lemma} \label{lemma kernel lin}
% The $\Primes$-local Burnside ring splits as
% %
% \[ \AP{G} \cong e_{cyc} \cdot \AP{G} \oplus e_{ker} \cdot \AP{G} \]
% %
% where $e_{cyc}$ (respectively $e_{ker}$) is defined to be the sum of all primitive idempotents $e_L$ with
% $L$ cyclic (respectively non-cyclic). Moreover, the summand $e_{ker} \cdot \AP{G}$ is precisely the kernel of
% the linearization map $\lin \colon \AP{G} \to \RXP{G}$.
% \end{lemma}

\begin{proof}[Proof of Theorem~\ref{thm rep splitting norms}]
We only need to show the equivalence (c) $\Leftrightarrow$ (d).
By Theorem~\ref{thm alg preservation} and Lemma~\ref{lemma little trick}, the statement (c) (respectively (d))
holds if and only if the equation
\[ N^H_K R^G_K(x) \cdot R^G_H(x) = R^G_H(x) \]
holds in $\AP{H}$ for $x = e_C$ (respectively in $\RXP{H}$ for $x = \lin(e_C)$). The linearization map $\lin
\colon \AP{-} \to \RXP{-}$ is a map of Tambara functors, hence preserves norms, restrictions and multiplication. By
Lemma~\ref{lemma kernel lin}, $\lin$ is injective on the ideal summand $e_{cyc} \cdot \AP{G}$ and that summand
contains the element $e_C$. It follows that the above equation holds for $x = e_C$ if and only if it holds
for $x = \lin(e_C)$.
\end{proof}

We can use the language of incomplete Tambara functors \cite{BH:ITF} to describe the algebraic structure of
$\Loc{\RXP{G}}{e_C}$ in terms of certain indexing systems.

\begin{prop}[\cite{boehme:mult-idempot}, Prop.~4.16] \label{prop indexing system IL}
Let $L \leq G$ be $\Primes$-perfect.
There is an indexing system $\IL$ given as follows:
for all $H \leq G$, $\IL(H)$ is the full subcategory of finite $H$-sets spanned by
all coproducts of the orbits $H/K$ such that
the groups $K \leq H \leq G$ satisfy
% statement (e) of Theorem~\ref{intro summary thm} \wrt $L$.
the following condition:
Any subgroup $L' \leq H$ conjugate in $G$ to $L$ lies in $K$.
\end{prop}

Note that this condition on $H/K$ is a generalized version
of condition (e) of Theorem~\ref{intro summary thm}
with $C$ and $C'$ replaced by $\Primes$-perfect subgroups $L$ and $L'$ that are not necessarily cyclic.

\begin{thm} \label{thm rep splitting ITF}
Let $C \leq G$ be a cyclic $\Primes$-perfect subgroup, and denote by $\RXP{-}$ one of the Tambara functors
$\RUP{-}$ or $\ROP{-}$. Then the following hold:
\begin{enumerate}[i)]
  \item The Green functor $\Loc{\RXP{-}}{e_C}$ admits the structure of an $\IC$-Tambara functor under
  $\RXP{-}$.
  \item The indexing system $\IC$ is maximal among the indexing systems that satisfy i).
  \item The canonical map $\RXP{-} \to \Loc{\RXP{-}}{e_C}$ is an $e_C$-localization in the
  category of $\IC$-Tambara functors. \qed
\end{enumerate}
\end{thm}

We record two easy consequences of our characterization of norm maps in the idempotent summands.

\begin{cor} \label{cor alg principal block}
The summand $\Loc{\RXP{-}}{e_C}$ is a Tambara functor (i.e., has a complete set of norms) if and only if
$C$ is the trivial group. \qed
\end{cor}

\begin{cor} \label{cor alg normal}
The subgroup $C$ is normal in $G$ if and only if the summand $\Loc{\RXP{-}}{e_C}$ admits all norms of the form
$\tilde{N}_K^H$ such that $K$ contains a subgroup conjugate in $G$ to $C$. \qed
% The following are equivalent:
% \begin{enumerate}[a)]
%   \item The summand $\Loc{\RXP{-}}{e_C}$ admits all norms of the form $\tilde{N}_K^H$ such that $K$ and $H$
%   both contain a subgroup conjugate in $G$ to $C$.
%   \item The subgroup $C$ is normal in $G$.
% \end{enumerate}
\end{cor}

% \begin{rem}
% BONUS: splitting result for $\RUP{G} \otimes \sint$
% \end{rem}

%+++++++++++++++++++++++++++++++++++++++++++++++++++++++++
\subsection{Multiplicativity of the idempotent splittings} \label{subsect:alg mult splitting}
We can now describe the multiplicativity of the idempotent splitting of $\RXP{-}$ in terms of the
indexing system
\[ \Icyc := \bigcap_{(C)} \IC \]
arising as the intersection of the indexing systems $\IC$ defined in Prop.~\ref{prop indexing system IL}.

\begin{prop} \label{prop alg splitting}
The localization maps $\RXP{-} \to \Loc{\RXP{-}}{e_C}$ assemble into an isomorphism of $\Icyc$-Tambara
functors
\[ \RXP{-} \to \prod_{(C) \leq G} \Loc{\RXP{-}}{e_C} \]
where the product is taken over conjugacy classes of cyclic $\Primes$-perfect subgroups.
Moreover, $\Icyc$ is maximal among all indexing sets with this property.
\end{prop}

\begin{proof}
Each of the canonical maps $\RXP{-} \to \Loc{\RXP{-}}{e_C}$ is a map of $\IC$-Tambara functors by \ref{thm rep splitting ITF},
hence their product is a map of $\Icyc$-Tambara functors. It is a levelwise isomorphism by construction. The
maximality also follows from Theorem~\ref{thm rep splitting ITF}: it implies that $\Icyc$ is maximal among the
indexing systems $\cJ$ such that each summand $\Loc{\RXP{-}}{e_C}$ is a $\cJ$-Tambara functor.
\end{proof}

The admissible sets of $\Icyc$ can be characterized as follows.

\begin{lemma}[\cite{boehme:mult-idempot}, Lemma~4.23] \label{admissible sets of I}
Let $K \leq H \leq G$, then $H/K$ is an admissible set for $\Icyc$ if and only if for all cyclic $\Primes$-perfect
$C \leq H$, $C$ is contained in $K$.
\end{lemma}

%Section: Homotopical results
\section{Idempotent splittings of equivariant K-theory} \label{section:htpy}
Let $\KXG$ denote one of the genuine $G$-spectra $\KUG$ or $\KOG$, i.e., either complex or real equivariant
$K$-theory. We will determine the multiplicativity of the $\Primes$-local idempotent splitting
\[ \KXGP \simeq \prod_{(C)} \Loc{\KXGP}{e_C}, \]
i.e., we will explicitly describe the maximal $N_\infty$ algebra structure on each of the factors, as well as
the maximal $N_\infty$ algebra structure preserved by the splitting.
% Let $C \leq G$ be a cyclic $\Primes$-perfect subgroup.
Recall that as a consequence of Theorem~\ref{intro thm idempotents
RU}, the blocks of $\KXGP$ are given as the $e_C$-localizations 
\[ \KXGP \wedge \Loc{\sphere_{(\Primes)}}{e_C} \]
of $\KXGP$ in the category of $G$-spectra.

%+++++++++++++++++++++++++
\subsection{Preliminaries} \label{subsect:prelim}
The $N_\infty$ operads of \cite{BH:OperMult} structure $G$-equivariant ring spectra with incomplete sets of
norm maps parametrized by their associated indexing systems.
% The realization problem \cite[Cor.~5.6,
% Conj.~5.11]{BH:OperMult} asks whether we can find an $N_\infty$ operad (necessarily unique up to equivalence)
% for any given indexing system. Independent positive solutions to the problem were given in
% \cite[Thm.~4.7]{GWv3}, \cite[Thm.~3.3]{rubin:realization} and \cite[Cor.~IV]{bonventre-pereira}.
% In particular, we can realize the indexing systems $\IC$ as well as $\Icyc = \cap_{(C)} \IC$ as operads.
According to \cite[Thm.~4.7, Prop.~4.10]{GW}, any given indexing system can be realized as the indexing
system of a $\Sigma$-cofibrant\footnote{An operad $\cO$ in $G$-spaces is $\Sigma$-\emph{cofibrant} if each space $\cO(n)$
is of the homotopy type of a $(G \times \Sigma_n)$-CW complex.} $N_\infty$ operad. Similar existence results were given in
\cite[Thm.~2.16]{rubin:realization} and \cite[Cor.~IV]{bonventre-pereira}.

\begin{notation}
For each conjugacy class of cyclic $\Primes$-perfect subgroups $C \leq G$, let $\OC$ be a $\Sigma$-cofibrant
$N_\infty$ operad whose associated indexing system is $\IC$. Let $\Ocyc$ be a $\Sigma$-cofibrant $N_\infty$
operad whose associated indexing system is $\Icyc$.
\end{notation}

Note that by definition, an $N_\infty$ operad $\cP$ is a certain operad in the category of unbased $G$-spaces.
By the usual abuse of notation, we refer to an algebra over the operad $\Sigma^\infty_+ \cP$ in $G$-spectra as
a $\cP$-\emph{algebra}.

\begin{rem} \label{remark comm}
For any choice of the operad $\OC$, both $\sphere$ and $\KXG$ are naturally algebras over $\OC$: both spectra
can be modelled as strictly commutative monoids in orthogonal $G$-spectra, and hence admit an action by $\OC$
that factors through the action of the commutative operad.
\end{rem}

% \begin{rem} \label{remark replacement}
% For any choice of the operad $\OC$, the $G$-equivariant sphere spectrum $\sphere$ is canonically equivalent to
% an algebra over $\OC$: The axioms of an $N_\infty$ operad force the zeroth space $\OC(0)$ to be weakly
% contractible, hence $\Sigma^\infty_+ \OC(0) \simeq \sphere$. The $G$-equivariant $K$-theory spectrum spectrum
% $\KXG$ can be given the structure of a commutative monoid in orthogonal $G$-spectra,
% cf.~e.g.~\cite[Constr.~6.4.9]{schwede:GHT}. Then $\KXG$ is canonically equivalent to $\OC(0)_+ \wedge \KXG$,
% which becomes an $\OC$-algebra via
% % \redfont{(TODO: check algebra axioms.)}
% %
% \[ \OC(n)_+ \wedge (\OC(0)_+ \wedge \KXG)^{\wedge n} \cong \OC(n)_+ \wedge \OC(0)_+^{\wedge n} \wedge
% \KXG^{\wedge n} \to \OC(0)_+ \wedge \KXG \]
% %
% where the latter map uses the operad action on $\OC(0)_+$ and the multiplication of $\KXG$. \redfont{(MH:
% just factor through $Comm$, which is much easier.)}
% \end{rem}

%++++++++++
\subsection{Multiplicativity of the idempotent summands} \label{subsect:htpy mult summands}
We are now ready to state our main homotopical result.

\begin{thm} \label{thm htpy}
Let $C \leq G$ be a cyclic $\Primes$-perfect subgroup.
% Let $\KXGP$ denote either $\KUGP$ or $\KOGP$.
Then:
\begin{enumerate}[i)]
  \item The $G$-spectrum $\Loc{\KXGP}{e_C}$ is an $\OC$-algebra under $\KXGP$.
  \item The operad $\OC$ is maximal among the $N_\infty$-operads that satisfy i).
  \item The canonical map $\KXGP \to \Loc{\KXGP}{e_C}$ is an $e_C$-localization in
  the category of $\OC$-algebras in $G$-spectra.
\end{enumerate} 
\end{thm}

The key to the proof is the following preservation result for $N_\infty$ algebras given in
\cite{boehme:mult-idempot}. It extends previous work of Hill and Hopkins \cite{HH:EqvarMultClosure} and
uses a result of Guti\'errez and White \cite[Cor.~7.10]{GW}.

\begin{prop}[\cite{boehme:mult-idempot}, Prop.~2.32] \label{prop translation White HH}
Let $\cP$ be a $\Sigma$-cofibrant $N_\infty$ operad. Fix $x \in \pi_0^G(\sphere_{(\Primes)})$. Then the
Bousfield localization $L_{x}$ given by smashing with
\[ \Loc{\sphere_{(\Primes)}}{x} =
\hocolim \left( \sphere_{(\Primes)} \stackrel{x}{\longrightarrow} \sphere_{(\Primes)}
\stackrel{x}{\longrightarrow} \ldots \right) \]
preserves\footnote{ in the sense of \cite[Def.~7.3]{GW}} $\cP$-algebras in $\Primes$-local $G$-spectra if
and only if for all $H \leq G$ and all transitive admissible $H$-sets $H/K$, the element
$N_K^H R^G_K (x)$
divides a power of $R^G_H(x)$ in the ring $\pi_0^H(\sphere_{(\Primes)})$.
\end{prop}

\begin{proof}[Proof of Theorem~\ref{thm htpy}]
\emph{Ad i):} We know from Theorems~\ref{thm alg preservation} and~\ref{thm rep splitting norms} that for
each of the admissible sets of $\IC$, hence of $\OC$, the division relation of Prop.~\ref{prop translation
White HH} holds, so $\Loc{\KXGP}{e_C}$ is an $\OC$-algebra under $\KXGP$. \\
\emph{Ad ii):} Assume that $\cP$ is an element strictly greater than $\OC$ in the poset of (homotopy types
of) $N_\infty$ operads. Then any norm that comes from $\cP$ but not from $\OC$ induces a corresponding norm on
homotopy groups that does not correspond to an admissible set of $\IC$, thus contradicting the maximality
statement included in Theorem~\ref{thm rep splitting ITF}.
\\
\emph{Ad iii):} It is an $e_C$-localization in $G$-spectra and a map of $\OC$-algebras.
\end{proof}

We obtain the homotopical analogue of Corollary~\ref{cor alg principal block}, stated as Corollary~\ref{intro
cor principal block KU} in the introduction. There is also a homotopical version of Corollary~\ref{cor alg normal}:

% \begin{cor}
% The summand $\Loc{\KXGP}{e_C}$ is a $G$-$E_\infty$ ring spectrum (i.e., has a complete set of
% Hill-Hopkins-Ravenel norm maps) if and only if $C$ is the trivial group. \redfont{(TODO: Already stated in
% intro.)}
% \end{cor}

\begin{cor}
The group $C$ is normal in $G$ if and only if $\Loc{\KXGP}{e_C}$ admits all norm maps of the form
$\tilde{N}_K^H$ such that $K$ and $H$ both contain a subgroup conjugate in $G$ to $C$. \qed
% The following are equivalent:
% \begin{enumerate}[a)]
%   \item The summand $\Loc{\KXGP}{e_C}$ admits all norm maps of the form $\tilde{N}_K^H$ such that $K$ and $H$
%   both contain a subgroup conjugate in $G$ to $x$.
%   \item The group $C$ is normal in $G$.
% \end{enumerate}
\end{cor}

%++++++++++++++++++++++++++++++++++++++++++++++++++++++++
\subsection{Multiplicativity of the idempotent splitting} \label{subsect:htpy mult splitting}
We can also describe the multiplicativity of the entire idempotent splitting:

\begin{cor} \label{local cor htpy splitting}
Let $\Ocyc$ be a $\Sigma$-cofibrant $N_\infty$ operad realizing the indexing system $\Icyc = \bigcap_{(C)} \IC$.
Then the idempotent splitting
\[ \KXGP \simeq \prod_{(C)} \Loc{\KXGP}{e_C} \]
is an equivalence of $\Ocyc$-algebras. Here, the product is taken over all
conjugacy classes of cyclic $\Primes$-perfect subgroups of $G$.
\end{cor}

%Secret Appendix: Overview of Literature
% \input{literature.tex}

\phantomsection
\addcontentsline{toc}{section}{Bibliography}
\bibliographystyle{alpha}
{\footnotesize \bibliography{references} }

\end{document}